\documentclass{amsart}
\usepackage{amsfonts,amscd, amssymb, dsfont, mathrsfs}

\newtheorem{theorem}{Theorem}[section]
\newtheorem{lemma}[theorem]{Lemma}
\newtheorem{corollary}[theorem]{Corollary}
\newtheorem{proposition}[theorem]{Proposition}
\theoremstyle{remark}

\theoremstyle{definition}
\newtheorem{definition}[theorem]{Definition}

\numberwithin{equation}{section}
\makeatother

\newcommand{\I}{{\mathds {1}}}

\DeclareMathOperator{\B}{{\mathcal B}}

\newcommand{\A}{A}
\newcommand{\M}{M}
\newcommand{\cM}{\mathfrak{M}}
\newcommand{\D}{D}

\newcommand{\clo}[1]{\langle{#1}\rangle}

\begin{document}

\title{On a class of subdiagonal algebras} 
\author{David P. Blecher}
\address{Department of Mathematics, University of Houston, Houston, TX
77204-3008, USA}
\email[David P. Blecher]{dpbleche@central.uh.edu}

\author{Louis E. Labuschagne}
\address{Internal Box 209, School of Comp., Stat. \& Math. Sci., NWU, Pvt. Bag X6001, 2520 Potchefstroom, South
Africa}
\email{louis.labuschagne@nwu.ac.za}

\date{today}  
\thanks{We are grateful to  the Simons Foundation
for financial support. } 

\begin{abstract}  We investigate some new classes  of  operator algebras which we call  semi-$\sigma$-finite subdiagonal and 
Riesz approximable.  
These  constitute the most general setting to date for  a noncommutative Hardy space theory based on Arveson's subdiagonal algebras. 
We  develop this theory and study the properties of these new classes.    
  \end{abstract}

\maketitle

\section{Introduction}

From its inception in the remarkable thesis and 1967 paper of Arveson \cite{Arv}, the theory of subdiagonal algebras has now developed into a very refined theory of quantum $H^p$ spaces which has attracted very many international researchers.   Primary  structural facts regarding classical $H^p$ spaces and their abstract function theoretic generalizations from the 1960's,
now have beautiful quantum analogues in this setting.  We mention for example
the  F $\&$ M Riesz theorem, Beurling  invariant subspace theorem, Jensen and Szeg\"o theorems, Gleason-Whitney theorem, inner-outer factorization, and so on. 
The quantum theory is moreover no mere verbatim clone of the classical theory. Rather one often finds  in the quantum world that a theory which is fairly simple to formulate in the classical setting,  unfolds into a very intricate kaleidoscope in the quantum world.  

In a series of papers the authors extended 
the theory of generalized
$H^p$ spaces for function algebras from the 1960s
to the setting of Arveson's {\em finite maximal subdiagonal
algebras} of a von Neumann algebra $M$ possessing a faithful
normal tracial state.   As stated in \cite{BLsurvey}, ``as an example of what some might call `mathematical quantization', or noncommutative (operator algebraic) generalization of a classical theory, the program succeeds to a degree of `faithfulness to the original' which seems to be quite rare''.  Much of this theory is summarized in the the latter survey
paper.   
After this,  many researchers turned their attention to generalizing parts of the theory to subalgebras of semi- and $\sigma$-finite von Neumann algebras. Important structural results were obtained by Ji, Ohwada, Saito, Bekjan, Xu, the authors, and others, 
although other results
start to break down in successively more general settings.  In the present paper we develop the theory in the 
most general setting to date.
We propose  new classes  of algebras, that we call semi-$\sigma$-finite subdiagonal  and
Riesz  approximable, that contain both the semi- and $\sigma$-finite case.  We then work out  
some of the generalized Hardy space theory for these algebras. 
The essence of what we present is a theory of subdiagonal subalgebras conditioned to von Neumann algebras equipped with a strictly semifinite faithful normal weight.  
In fact an extension of the theory to general von Neumann algebras is possible, as will be shown in the forthcoming paper \cite{LXdraft}. However we believe that the present setting may well be the most general setting in which some of the more delicate properties of noncommutative $H^p$-spaces - like the F \& M Riesz theorem and the Gleason-Whitney theorem - hold true. As such this context deserves to be studied in its own right.

We will see that if $A$ is a semi-$\sigma$-finite subdiagonal algebra in $M$ then $A$ is an increasing limit of 
 maximal $\sigma$-finite subdiagonal algebras $A_i$.   This allows  a very convenient simplification of Haagerup's reduction method, and its application to
 subdiagonal algebras.  The Hardy space 
$H^p(A)$ is simply the completion of the 
increasing union of the Hardy spaces $H^p(A_i)$.   The theory of maximal subdiagonal algebras in the $\sigma$-{\em finite case} has been 
developed by many authors: Ji, Saito, Labuschagne, 
Bekjan and others (see e.g.\ \cite{JiSaito,JOS,Xu,Ji,JiAnalytic,LL,blueda,BekR}).   Thus most of the Hardy space theory of maximal 
semi-$\sigma$-finite subdiagonal algebras may be viewed as a `limiting case' of the 
theory of maximal $\sigma$-finite subdiagonal algebras. 
This is similar to the way in which most of the theory of maximal {\em semifinite} subdiagonal algebras,  developed by Bekjan, 
Sager and others, is a `limiting case' of the 
theory of maximal (tracially) finite subdiagonal algebras (see e.g.\ \cite{Bek,BO,sager}).  

We summarize briefly the structure of our paper. 
In Section \ref{defssf} we define  semi-$\sigma$-finite subdiagonal algebras, and give  some alternative characterizations and examples. 
In Sections \ref{strt} and \ref{Beu}, we  work out some aspects of the basic  Hardy space theory for these algebras, 
including some Beurling  invariant subspace theory.  In Section \ref{Rapp} we introduce a much larger class of algebras that we call {\em Riesz  approximable}, and 
check that the Lebesgue decomposition, F \& M Riesz, and Gleason-Whitney theorems all hold for these.  More generally we generalize several 
of our results from \cite{BLue,blueda} related to the latter theorems to a wider setting.  Some of this is related to the recent paper \cite{ClouatreH}. 

We now turn to notation and some background facts.  We refer the reader to 
\cite{Arv} and our survey \cite{BLsurvey} for the basic ideas, and for the `tracially finite' variant of the theory. 
Suppose  that $M$ is a von Neumann algebra with a faithful normal semifinite weight $\omega$, and modular group $(\sigma_t^\omega)$.

\begin{definition}
A weak* closed unital subalgebra $A$ of $M$ is said to be an \emph{analytically conditioned} subalgebra if  
\begin{itemize}
\item $\sigma_t^\omega(A)=A$ for all $t\in \mathbb{R}$, 
\item $\omega_{|D}$ is semifinite where $D=A\cap A^*$,
\item the faithful normal conditional expectation $\mathbb{E}:\M\to \D=A\cap A^*$ satisfying $\omega\circ\mathbb{E}=\omega$ (ensured by the above condition by
\cite[Theorem IX.4.2]{tak2}), is multiplicative on $A$.
\end{itemize}
\end{definition}

A  subalgebra $A$ of $M$ is said to be 
 {\em subdiagonal with respect to} $\omega$ if in addition to 
being analytically conditioned, it also satisfies the requirement that $A+A^*$ is weak* dense in $M$. 
We say that a subdiagonal  algebra is  maximal subdiagonal (with respect to
$\mathbb{E}$) if  it is not properly contained in
any larger proper subdiagonal algebra in $M$ with respect to ${\mathbb E}$.
For a subspace $F$ of $M$ we  define $F_0$ as usual to be $F \cap {\rm Ker}(\mathbb{E})$, as is 
 usual in the subdiagonal theory.
 
Our analysis requires some background on noncommutative $L^p$-spaces (see e.g.\ \cite{HJX,GLa1,GLa2}). We shall of needs be brief. Full details of all claims made below may be found in \cite{GLa1}. Throughout this discussion $\omega$ will be a fixed faithful normal semifinite reference weight on a von Neumann algebra $M$. The theory is a lot more accessible in the case where the reference weight is tracial (that is $\omega(a^*a)=\omega(aa^*)$ for each $a\in M$). It is convention to denote the reference weight by $\tau$ in the tracial case. In the tracial case the algebra $M$ may be enlarged to the so-called algebra of $\tau$-measurable operators $\widetilde{M}$ which is defined to be the set of all densely defined closed operators $f$ affiliated to $M$ which satisfy the requirement that for some $\epsilon>0$ we have that $\tau(\chi_{(\epsilon,\infty)}(|f|))<\infty$. This enlarged space turns out to be a complete metrizable algebra which is large enough to contain all the $L^p$-spaces and which admits an extension of $\tau$ to $\widetilde{M}$. Given any $1\leq p<\infty$ the space $L^p(M,\tau)$ is then simply defined to be $L^p(M,\tau)=\{f\in \widetilde{M}: \tau(|f|^p)<\infty\}$, with the norm given by $\|f\|_p= \tau(|f|^p)^{1/p}$. 

The construction of $L^p$-spaces for general von Neumann algebras is much more challenging. We shall follow Haagerup's approach to construct these spaces. The first step in this approach is to use the modular automorphism group $\sigma_t^\omega$ ($t\in \mathbb{R}$) induced by $\omega$ to construct the crossed product $M\rtimes_\omega\mathbb{R}$. Here the dual action of $\mathbb{R}$ on this crossed product is induced by a group of automorphisms $(\theta_s)$ on $M\rtimes_\omega\mathbb{R}$ satisfying
\begin{equation}\label{5:eqn dual-R} 
\theta_s(\pi(a))=\pi(a) \mbox{ and }\theta_s(\lambda_t)= e^{-ist}\lambda_t \mbox{ for each } a\in M\mbox{ and }t,s\in \mathbb{R}.
\end{equation}\label{eq:dualmodgp}
With $\widetilde{\omega}$ denoting the dual weight on $M\rtimes_\omega\mathbb{R}$, it is moreover possible to show that  
\begin{equation}\sigma_t^{\widetilde{\omega}}(f)=\lambda_tf\lambda_t^*\mbox{ with } \sigma_t^{\widetilde{\omega}}(\pi(a))=\pi(\sigma_t^\omega(a))\mbox{ for all }f\in(M\rtimes_\omega\mathbb{R})\mbox{ and all }a\in M.
\end{equation} 
By Stone's theorem there exists a nonsingular positive operator $h$ affiliated to $M\rtimes_\omega\mathbb{R}$ for which we have that $\lambda_t=h^{it}$ for all $t\in \mathbb{R}$. The fact that $\sigma_t^{\widetilde{\omega}}$ is implemented 
as above ensures that 
$M\rtimes_\omega \mathbb{R}$ is in fact semifinite, with the prescription $\tau(\cdot)=\widetilde{\omega}(h^{-1}\cdot)$ yielding an fns (that is, 
faithful normal semifinite)
trace on $M\rtimes_\omega\mathbb{R}$ for which we have that $\tau\circ\theta_s=e^{-s}\tau$ for all $s\in \mathbb{R}$. 
So by construction $h$ is just the Pedersen-Takesaki Radon-Nikodym derivative $\frac{d\widetilde{\omega}}{d\tau}$ of $\widetilde{\omega}$ with respect to $\tau$ \cite{PT}. The semifiniteness ensures that $\cM=M\rtimes_\omega\mathbb{R}$ may be enlarged to the algebra $\widetilde{\cM}$ of $\tau$-measurable operators, with each $\theta_s$ in addition extending continuously to 
this enlarged algebra. For each $1\leq p<\infty$ the Haagerup $L^p$-space is then defined to be the space $L^p(M)=\{a\in \widetilde{\cM}: \theta_s(a)=e^{-s/p}a\mbox{ for all }s\in\mathbb{R}\}$. It is known that $L^\infty(M)$ corresponds to the canonical copy of $M$ inside $M\rtimes_\omega \mathbb{R}$.  The space $L^1(M)$ admits a so-called tracial 
functional $tr$, which can be used to realise the norm on $L^p(M)$ by means of the prescription $\|a\|_p=tr(|a|^p)^{1/p}$.

In the case where $e$ is a projection in the centralizer of $\omega$ with $\omega$ semifinite on $eM e$, one may canonically identify $L^p(e M e)$ (constructed using the restriction of $\omega$ to $eM e$) with $eL^p(M)e$. What lies behind this, is the fact that the assumption that $e$ is a fixed point of the modular group, allows one to canonically identify $e(M\rtimes_\omega\mathbb{R})e$ with $(eMe)\rtimes_{\omega_e}\mathbb{R}$ where $\omega_e=\omega_{|eMe}$. Details of the computation may be found in the discussion following Lemma 8 of \cite{Wat} or in \cite{GLa2}.

Each $1\leq p<\infty$ moreover admits a norm-dense bijective embedding $\mathfrak{i}^{(p)}:\mathfrak{m}_\omega\to L^p(M)$ of the algebra $\mathfrak{m}_\omega$ into $L^p(M)$. Formally this embedding may be thought of as a prescription taking the form 
 $a\to h^{1/2p}ah^{1/2p}$. Since here $h$ may not be $\tau$-measurable, there are however significant challenges regarding existence and closability which need to be overcome, regarding which we shall not elaborate here. Details of the construction of these embeddings may be found in Section 2 of \cite{GL2}. In the case considered above where $e$ is a projection in the centralizer of $\omega$, the fact that $\sigma_t^\omega$ is implemented by the unitary group $h^{it}$ inside the crossed product will, when combined with the fact that $e$ is a fixed point of the modular group,  ensure that $e$ commutes with each $h^{it}$ and hence strongly with $h$ itself. On tracing the construction of the embeddings $\mathfrak{i}^{(p)}$ in \cite{GL2}, this fact in turn ensures that we will  have that $\mathfrak{i}^{(p)}(eae)=e\mathfrak{i}^{(p)}(a)e$ for any $a\in \mathfrak{m}_\omega$. 

One interesting fact which we shall have occasion to apply to conditional expectations, concerns positive maps $T:M\to M$ satisfying the condition $\omega\circ T\leq \omega$. For such maps we clearly have that $T(\mathfrak{m}_\omega)\subset \mathfrak{m}_\omega$. However as can be seen from \cite{HJX}, 
the prescription  $T^{(p)}(\mathfrak{i}^{(p)}(a))=\mathfrak{i}^{(p)}(T(a))$ ($a\in \mathfrak{m}_\omega$) yields a $\|\cdot\|_p$-norm bounded operator uniquely extending to a bounded operator on $L^p(M)$.

It is clear from the definition of these Haagerup $L^p$ spaces that they each have a different ``phase''. So to makes\ sense of concepts like intersections of these spaces in a manner which harmonises with the classical setting, these spaces first need to be embedded into a common superspace where their orientation with respect to each other is in line with that of the classical setting. The superspace we
use is the noncommutative analogue of $L^1 + L^\infty$ constructed using the theory of noncommutative Orlicz spaces for general von Neumann algebras as espoused in \cite{GLa2, L-Orlicz}. The so-called fundamental function of $(L^1 + L^\infty)(\mathbb{R})$ is given by $\varphi_{1+\infty}(t)=\min(1,t)$ (where $t\geq 0$). We then use this function to define the space $L^{1+\infty}(M)$ by the prescription $$L^{1+\infty}(M) =\{a\in \widetilde{\cM}: \theta_s(a)=v_s^{1/2}av_s^{1/2} \mbox{ for all } s\in \mathbb{R}\},$$where each $v_s$ is a bounded element of $\cM$ given by $v_s=\varphi_{1+\infty}(h)\varphi_{1+\infty}(h)^{-1}$. (When the fundamental function of $L^p(\mathbb{R})$, namely $t\to t^{1/p}$, is substituted for $\varphi_{1+\infty}(h)$, we get exactly Haagerup's definition of $L^p(M)$.) When equipped with the topology inherited from $\widetilde{\cM}$, this space turns out to be a quasi-Banach space. 

It is known that each $L^p(M)$ admits a canonical embedding $\iota^{(p)}$ of $L^p(M)$ into $L^{1+\infty}(M)$ 
\cite[Proposition 7.22]{GLa1}. We write $\mathscr{L}^p$ for $\iota^{(p)}(L^p)$. The embedding of $M$ into $L^{1+\infty}(M)$ takes the 
form $\iota^{(\infty)}(a)=\varphi_{1+\infty}(h)^{1/2}a\varphi_{1+\infty}(h)^{1/2}$ for all $a\in M$, whereas for each $L^p(M)$ 
($1\leq p<\infty$) the embedding takes the form $\iota^{(p)}(a)=\eta_p(h)^{1/2}a\eta_p(h)^{1/2}$ for all $a\in L^p(M)$ where 
$$\eta_p(t) = \varphi_{1+\infty}(t)t^{-1/p} =\left\{\begin{array}{ll} t^{1/q} &\quad 0\leq t\leq 1\\ t^{-1/p} & \quad t>1\end{array} \right.$$In support of the claim that inside $L^{1+\infty}(M)$ the spaces $L^p(M)$ are appropriately oriented with respect to each 
other, we cite the fact that for any $1\leq p<\infty$ we have that $\iota^{(p)}(\mathfrak{i}^{(p)}\mathfrak{m}_\omega)= \iota^{(\infty)}(\mathfrak{m}_\omega)$. When working with the representations of $L^p(M)$ inside $L^{1+\infty}(M)$, we shall write 
$\mathscr{L}^p(M)$ for $\iota^{(p)}(L^p(M))$ equipped with the norm inherited from $L^p(M)$. We shall similarly write 
$\mathscr{H}^p(A)$ for $\iota^{(p)}(H^p(A))$ equipped with the norm inherited from $L^p(M)$

\section{Emergent paradigms: semi-$\sigma$-finite subdiagonal algebras} \label{defssf} 

Our framework is a von Neumann algebra $M$ and weak* continuous inclusions of weak* closed operator algebras  $D = A \cap A^*  \subset A \subset M$.
In this section we will consider a setting in which $M$ 
admits $\sigma$-finite weak* closed $*$-subalgebras $M_i$ whose union is weak* dense in $M$, 
with each $A_i = A \cap M_i$ a maximal subdiagonal algebra in $M_i$, 
and with $\bigcup_iA_i$ weak* dense in $A$.
Within this class a specific subclass now suggests itself:
 
\begin{theorem} \label{isma} Assume that $M$ is a von Neumann algebra with a faithful normal semifinite weight $\omega$,
with weak* continuous inclusions of weak* closed operator algebras  $D = A \cap A^*  \subset A \subset M$ as above. 
 In this setting, the following conditions are equivalent: 
 \begin{enumerate} 
\item $\A$ is a subdiagonal algebra in $\M$ with respect to $\omega$, 
and $\omega$ is strictly semifinite on $\D$; 
\item $\A$ is  
a subdiagonal algebra in $M$  with respect to $\omega$, 
and  there are $\omega$-finite projections $e_i \in D \cap  M_\omega$ with $e_i \nearrow 1$ (here $M_\omega$ is the centralizer of $\omega$ {\rm \cite{tak2}}); 
 \item  there are $\omega$-finite projections $e_i \in D \cap  M_\omega$ with $e_i \nearrow 1$ 
 such that $\A_i = e_i \A e_i$  
 is a maximal subdiagonal algebra in   the $\sigma$-finite von Neumann algebra 
$\M_i = e_i \M e_i$ with respect to the weight $\omega_{|\M_i}$. 
 \end{enumerate} 
 A subalgebra $\A$ satisfying these conditions  is a maximal subdiagonal algebra, and we have $\D \cap \M_\omega = \D_\omega$.
\end{theorem}

\begin{proof} (2) $\Rightarrow$ (3)\   We have 
$\D_i = e_i \D e_i \subset \A_i \subset \M_i = e_i \M e_i$, and $\A_i$ is subdiagonal in 
$\M_i$.   Indeed   $\A_i  + \A_i^* = e_i (\A+ \A^*)  e_i$ is weak* dense in $e_i \M e_i$, and we have 
$\A_i \cap \A_i^*  = e_i (\A \cap \A^*)  e_i  = \D_i$.     The restriction $\omega_i$ of $\omega$ to $(\M_i)_+$ is finite and extends to a faithful positive normal functional on $\M_i$, so 
$\M_i$ is  $\sigma$-finite. 
Note that $\omega_i$  has a modular automorphism group which is the restriction of $\sigma^\omega_t$ to $\M_i$, and this leaves $\A_i$ invariant.  
Also the restriction of $\mathbb{E}$  to $\M_i$ is $\omega_i$-preserving.
So by the main result in \cite{Xu} (and the remark after it), $\A_i$ is maximal subdiagonal in $\M_i$.

(3) $\Rightarrow$ (2)\   
The restriction $\omega_i$ of $\omega$ to $M_i$ has modular automorphism group which is the restriction of $\sigma^\omega_t$ to 
$\M_i$. As part of the condition in (3) that $\A_i$ is maximal subdiagonal with respect to $\omega_{|\M_i}$ we are 
assuming that there is a $\omega_{|\M_i}$-preserving conditional expectation 
$\mathbb{E}_i : \M_i \to \D_i$  which is multiplicative on $\A_i$. 
Then $\sigma^\omega_t$ leaves $\A_i$ and $\D_i$ invariant by the result of Ji, Ohwada and Saito cited in \cite{Xu} above the main theorem there. 
Hence  $\sigma^\omega_t$ leaves  $\A$ and $\D$ invariant. By Takesaki's theorem there exists a unique 
normal  $\omega$-preserving conditional expectation $\mathbb{E} : \M \to \D$, and by uniqueness its restriction
to $\M_i$ is $\mathbb{E}_i$.   It follows that  $\mathbb{E}$ is multiplicative on $\A$. (To see this note that the $A_i$'s are increasing.)  Finally, since $\cup_i \, \M_i$ is weak* dense in $\M$, and $\A_i  + \A_i^*$ is weak* dense in $\M_i$,  we have
 $\A$ subdiagonal in $\M$. 

(2) $\Rightarrow$ (1)\  This follows from Exercise VIII.2.1 of \cite{tak2} applied to the Claim: $\D \cap \M_\omega \subset \D_\omega$. To prove this claim we will use 
\cite[Theorem VIII.2.6]{tak2}. Suppose that  $x \in \D \cap \M_\omega$ and $y \in {\mathfrak m}^{\D}_\omega$, with $y = n_1^* n_2$ say for 
$n_i \in {\mathfrak n}^{\D}_\omega = \D \cap  {\mathfrak  n}^{\M}_\omega$. Since by \cite[Theorem VIII.2.6]{tak2} we have that $n_2^* n_2 x \in {\mathfrak m}^{\M}_\omega$ and hence that $x^* n_2^* n_2 x  \in {\mathfrak m}^{\M}_\omega$, it is clear that $n_2 x \in {\mathfrak n}^{\D}_\omega$. Thus $y x = n_1^* n_2x \in  {\mathfrak m}^{\D}_\omega$.   Similarly 
$xy \in  {\mathfrak m}^{\D}_\omega$. Hence we clearly have that $x \in \D_\omega$.

 (1) $\Rightarrow$ (2)\   Since by definition 
 $\sigma^\omega_t$ leaves  $A$  invariant, it clearly leaves $D$ invariant, so that 
the restriction of $\sigma^\omega_t$ to  $D$ is the automorphism group of $\omega_{| D}$.  
 So $\D \cap \M_\omega = \D_\omega$ and using one of the well known characterizations of  strictly semifinite weights 
(see e.g.\ \cite{GLa1,GLa2}) we obtain an increasing net $e_i  \nearrow 1$ as in (2).

To see that $\A$ is indeed maximal suppose that $\A \subset \B \subset \M$ where $\B$ is a 
subdiagonal algebra with respect to  $\mathbb{E}$.   In particular $\B \cap \B^* = \D$.  Then $\D_i \subset \A_i \subset \B_i = e_i \B e_i  \subset \M_i$,
and $\B_i$ is easily seen as in the last proof to be subdiagonal in the $\sigma$-finite von Neumann algebra
$\M_i$.     So $\A_i = \B_i$ by the main theorem in \cite{Xu}.  Thus  any $x \in \B$ has $x = \lim_i \, e_i x e_i \in \A$.  So $\A$ is a maximal subdiagonal algebra. 
\end{proof}

\begin{quote} An algebra $\A$ satisfying all of the above conditions  will be called a \emph{semi-$\sigma$-finite subdiagonal algebra}.\end{quote} 
  This class has essentially appeared very briefly in the literature, on the last page of \cite{Xu} where it is defined in terms
  of a strictly semifinite weight 
  (to see the connection with (1) one 
  may use  the remark after \cite[Theorem 1.1]{Xu}).
   A different route to the maximality of such $A$ was stated there with hints. 
It is true but not at present clear that all of the algebras 
 considered on the last page of  \cite{Xu} are in fact semi-$\sigma$-finite subdiagonal algebras; this needs a generalization of 
 \cite[Theorem 2.4]{JOS} that will appear in forthcoming work of the second author.

Semi-$\sigma$-finite subdiagonal algebras generalize maximal subdiagonal algebras in   both semifinite and in $\sigma$-finite von Neumann algebras:
 
 \begin{proposition} \label{isma2} A maximal subdiagonal algebra in  a semifinite or a $\sigma$-finite von Neumann algebras  is
 a  semi-$\sigma$-finite subdiagonal algebra.
\end{proposition}

\begin{proof}   
  In the $\sigma$-finite case  we are considering the maximal subdiagonal algebra  studied in e.g.\ \cite{JOS}.
Here $\omega$ is finite, so strictly semifinite on $D$, and we  may take $e_i = 1$ in the definition of a semi-$\sigma$-finite subdiagonal algebra.  
In the semifinite  case we are assuming (as in the papers on subdiagonal algebra in  the semifinite case referred to above) 
that  $\omega$ is a semifinite trace on both $M$ and $D$ and that $A$ is subdiagonal with respect to $\omega$.
Here  $\sigma^\omega_t$ is the
identity map.   The existence of   $\omega$-finite projections $e_i \in D = D \cap  M_\omega$ with $e_i \nearrow 1$ is immediate from the semifiniteness.   \end{proof}

We end this section by mentioning further examples of  semi-$\sigma$-finite subdiagonal algebras.
 Every von Neumann algebra $M$ is a semi-$\sigma$-finite subdiagonal algebra (so $A = M$).   This follows from 
the fact that every von Neumann algebra possesses a strictly semifinite faithful normal weight.   

If $A$ is a maximal $\sigma$-finite (or semifinite) subdiagonal algebra in $M$, then $A \bar{\otimes} B(H)$ is a   semi-$\sigma$-finite subdiagonal algebra.
Indeed for any von Neumann algebra $N$ we have that $A  \bar{\otimes} N$ is a  semi-$\sigma$-finite subdiagonal algebra (in $M \bar{\otimes} N$).
More generally the latter is also true if $A$ is a  semi-$\sigma$-finite subdiagonal algebra.  

To see these, we use several facts from e.g.\ \cite[Chapters 8 and 9]{Stratila}: recall that $M \bar{\otimes} N$ has a canonical faithful normal semifinite weight built from the two weights on $M$ and $N$.  Then  $A \bar{\otimes} N + A^* \bar{\otimes} N$ is certainly weak* dense in $M \bar{\otimes} N$.   That $(A \bar{\otimes} N) \cap  (A^* \bar{\otimes} N) = D \bar{\otimes} N$ follows easily 
from Tomiyama's `Fubini' slice map theorem: any $x \in (A \bar{\otimes} N) \cap  (A^* \bar{\otimes} N)$ has its left slices $L_\psi(x) \in A \cap A^* = D$ for $\psi \in N_*$. 
If $E$ is the weight preserving expectation onto $D = A \cap A^*$
then $E \otimes I_N$ is a weight preserving normal conditional expectation onto $D \bar{\otimes} N$, and it is multiplicative
on $A \bar{\otimes} N$.   The modular group for $M \bar{\otimes} N$ is 
$\sigma^M_t \otimes \sigma^N_t$, and this preserves $A \bar{\otimes} N$.  Finally suppose that $e_t \nearrow 1$ in $M$ and $f_s \nearrow 1$ in $M$ are in the centralizers
and are finite with respect to the respective 
weights.  Then $e_t \otimes f_s \nearrow 1 \otimes 1$, and $(\sigma^M_t \otimes \sigma^N_t)(e_t \otimes f_s) = e_t \otimes f_s$.  So $A \bar{\otimes} N$
is  a  semi-$\sigma$-finite subdiagonal algebra (in $M \bar{\otimes} N$).

\section{Basic structural theory of $H^p$ spaces of semi-$\sigma$-finite subdiagonal subalgebras} \label{strt} 
  
 As usual, for $1 \leq p < \infty$ 
we define $H^p(A)$ to be the closure in the $p$-norm of $\mathfrak{i}^{(p)}(\A \cap {\mathfrak m}^{\M}_\omega)$, and $H^p_0(\A)$ to be the closure in the $p$-norm of $\mathfrak{i}^{(p)}(\A_0 \cap {\mathfrak m}^{\M}_\omega)$. 
The utility of 
the following theorem is the approximation it affords. (See the comment following the theorem for details.)   
 
\begin{theorem} \label{decomp}  Suppose that we have weak* continuous inclusions of weak* closed operator algebras  $\D = \A \cap \A^*  \subset \A \subset \M$ ,
 and  $\M$ is a von Neumann algebra with a faithful normal semifinite weight $\omega$ and a normal  $\omega$-preserving conditional expectation $E : \M \to \D$.  We also assume that $\omega$ is  semifinite on $\D$, and moreover that there is an increasing net of  $\omega$-finite projections $e_i \in \D \cap \M_\omega$ with $e_i \nearrow 1$.  Let  $1 \leq p \leq \infty$, and let $\M_i, \A_i, \D_i$ be the usual compressions by $e_i$. 
 Then $L^p(\M_i) = e_i L^p(\M) e_i$ for each $i$, these are increasing with $i$, and the union of the $L^p(\M_i)$ 
 is dense in $L^p(\M)$.   Similarly $H^p(A_i)= [\mathfrak{i}^{(p)}(\A_i)]_p$ equals  $e_i H^p(\A) e_i$, and  these sets  are increasing with $i$, and  the $p$-closure 
of their union is $H^p(\A)$. Similarly $H^p_0(\A_i) = [\mathfrak{i}^{(p)}((\A_i)_0)]_p = e_i H^p_0(\A) e_i$, and these are increasing with the $p$-closure 
of their union being $H^p_0(\A) = H^p(\A) \cap {\rm Ker}(E_p)$.

In all of these increasing closures $F = \overline{\cup_i \, e_i F e_i}$ above, any $x \in F$ is a $p$-norm limit  $x = \lim_i \, e_i x e_i$, if $1 \leq p \leq \infty$
(a SOT limit if $p = \infty$).  \end{theorem}

\begin{proof} 
We said already that  $L^p(\M_i) = e_i L^p(\M) e_i$. 
 Of course $e_i$ acts on the left and right as contractive projections on $L^p(\M)$, and $e_i {\mathfrak n}^{\M}_\omega e_i \subset {\mathfrak n}^{\M_i}_\omega$. 
  Similarly, if $e_i \leq e_j$ then $L^p(\M_i) \subset L^p(M_j)$.   That the union of the $L^p(M_i)$ 
 is dense in $L^p(M)$ follows from the Claim: If  $e_i \nearrow 1$ in $M$ then 
for each $\xi \in L^p(M)$ both $(\xi e_i)$ and $(e_i\xi)$ converge to $\xi$ in $p$-norm.    To see this 
 we first note that $e_i \, \xi \to \xi$ weakly in $L^p(M)$.   Indeed this follows by definition of the  $L^p(M)$ and trace duality: if $\nu \in L^q(M)$ then $r = \xi \nu \in L^1(M)$ and  
 $${\rm tr} (\nu e_i \xi) = {\rm tr} (e_i r) \to  {\rm tr} ( r) =  
 {\rm tr} (\nu \xi) ,$$
 where tr is the tracial functional on $L^1$.   Similarly $\xi e_i \to \xi$ weakly.  Thus if $\xi e_i = 0$ (or $e_i \xi = 0$) for all $i$ then $\xi = 0$.
 It then follows from \cite[Lemma 2.4]{JS} that $\xi e_i \to \xi$ in $p$-norm.   By continuity of * in the $p$-norm (see  4.12(ii) and 7.12 in \cite{GLa1}) we have
 $e_i \xi \to \xi$ in $p$-norm.  Then $$\| e_i \xi e_i -  \xi \|_p \leq
 \| e_i ( \xi e_i -  \xi ) \|_p + \| e_i \xi -  \xi \|_p \to 0.$$
 
 We defined $H^p(A_i)$ to be $[\mathfrak{i}^{(p)}(A_i)]_p$ in $L^p(M_i)$.   Claim: $[\mathfrak{i}^{(p)}(A_i)]_p =e_i [\mathfrak{i}^{(p)}(A \cap {\mathfrak m}^M_\omega)]_p e_i$.   That $[\mathfrak{i}^{(p)}(A_i)]_p \subset  e_i [\mathfrak{i}^{(p)}(A \cap {\mathfrak m}^M_\omega)]_p e_i$ is easy to see,
and the  converse is clear since $e_i (A \cap {\mathfrak m}^M_\omega) e_i \subset  A_i$.    Similarly the $[\mathfrak{i}^{(p)}(A_i)]_p$ are increasing, and  the $p$-closure $F$ 
of their union is $[\mathfrak{i}^{(p)}(A \cap {\mathfrak m}^M_\omega)]_p$.  Indeed the $[\mathfrak{i}^{(p)}(A_i)]_p$ are clearly  contained in $[\mathfrak{i}^{(p)}(A \cap {\mathfrak n}^M_\omega)]_p$  and hence also $F$.
If $a \in A \cap {\mathfrak m}^M_\omega$ then $\mathfrak{i}^{(p)}(e_i a e_i) \in [ A_i ]_p$, and as we saw above
$\| \mathfrak{i}^{(p)}(e_i a e_i) - \mathfrak{i}^{(p)}(a) \|_p \to 0$.  So $[\mathfrak{i}^{(p)}(A \cap {\mathfrak m}^M_\omega]_p$ is contained in $F$. 
Claim: $(A_i)_0 = e_i A_0 e_i$.   Indeed clearly $e_i A_0 e_i \subset A_i \cap A_0 = (A_i)_0$.  Conversely,  $(A_i)_0 \subset A_i \cap A_0 \subset e_i A_0 e_i$.  
A similar argument to the above applied to  $(A_i)_0 = e_i A_0 e_i$ shows that we have $H^p_0(A_i) = [\mathfrak{i}^{(p)}((A_i)_0)]_p = e_i H^p_0(A) e_i$, and that these are increasing with the $p$-closure 
of their union being $[\mathfrak{i}^{(p)}(A_0 \cap {\mathfrak m}^M_\omega)]_p$.   \end{proof} 

The hypotheses of the theorem hold if $A$ is a semi-$\sigma$-finite subdiagonal subalgebra (note that the last proof
 did not use subdiagonality of $A$ or that $A$ is invariant under 
the modular group).
Throughout the rest of this part we will assume that $A$ is a semi-$\sigma$-finite subdiagonal subalgebra of a von Neumann algebra $M$ with respect to a given faithful normal semifinite weight
 $\omega$, in the sense described in Theorem \ref{isma}.

 The key point is that if $A$ is a  semi-$\sigma$-finite subdiagonal algebra in $M$ then the $A_i$ are maximal $\sigma$-finite subdiagonal algebra in $M_i$, 
with the last theorem then providing a very convenient simplification of Haagerup's reduction method and its application to general subdiagonal algebras.  In particular  
$L^p(M)$ is simply the completion of the increasing union vector space $\cup_i \, L^p(M_i)$  in the $p$-norm, and $H^p(A)$ is simply the completion of the 
increasing union of subdiagonal Hardy spaces $H^p(A_i)$.  As we said in the introduction, the $\sigma$-finite Hardy space theory has been  developed by 
many authors.  Using this technology, most of the theory of semi-$\sigma$-finite subdiagonal algebras will follow as a `limiting case' of the 
theory of maximal $\sigma$-finite subdiagonal algebras.    We give several examples of this approach:

 \begin{corollary}  \label{L2dens} {\rm ($L^2$-density) \ } For a semi-$\sigma$-finite subdiagonal algebra
 the closure of $H^2(A) + H^2(A)^*$ is $L^2(M)$. 
 \end{corollary} 

\begin{proof}   Clearly  in the setting of the previous result $L^2(M)$ contains the unions of the $[\mathfrak{i}^{(p)}(A_i + A_i^*)]_2$.
So for $L^2$-density in $L^2(M)$ in such a setting we just in addition need $L^2$-density in $L_2(M_i)$, i.e.\ that $\mathfrak{i}^{(p)}(A_i + A_i^*)$ is $2$-dense in $L_2(M_i)$.    This is well known for 
maximal subdiagonal algebras in a $\sigma$-finite von Neumann algebra such as our $M_i$.  
 \end{proof} 
 
 For a semi-$\sigma$-finite subdiagonal algebra we saw that we have $L^p(D) \subset L^p(M)$, and by duality it 
 follows easily that there exists an expectation $E_p : L^p(M) \to L^p(D)$ that 
 is modular over ${\mathfrak m}^D_\omega$.   Generally conditional expectations $L^p(M) \to L^p(D)$  have been studied in \cite{HJX,JX} and \cite{Gold},
 and the existence of $E_p$ in general is shown e.g.\ in  \cite{GLa2}.  We sketch a quick proof in 
 our setting: The proof of \cite[Lemma 2.2]{JX} shows that the expectation $E_p^i : L^p(M_i) \to L^p(D_i)$ is the 
 dual of the inclusion $L^p(D_i) \subset L^p(M_i)$.  Thus if $i : L^p(D) \to L^p(M)$ is the inclusion then by Banach space duality
 and the fact that $L^p(M_i) = e_i L^p(M) e_i$ it is easy to see that $i^*_{| L^p(M_i)}$ maps
 into the copy of $L^p(D_i)$ and agrees with $E_p^i$.   Thus $i^* : L^p(M) \to L^p(D)$ is an expectation, and by density it will have the desired  
 properties of the expectation (e.g.\ most of the assertions of \cite[Propositions 2.2 and 2.3]{JX}). 
 Also, $E_p = \mathbb{E}^{(p)}$ on $\mathfrak{i}^{(p)}({\mathfrak m}_\omega)$.

\begin{corollary}  \label{Bek3.2} For a semi-$\sigma$-finite subdiagonal algebra we have $L^p(D) \subset H^p(A)$, and the expectation $E_p : L^p(M) \to L^p(D)$ 
maps $H^p(A)$ onto $L^p(D)$.
\end{corollary} 

\begin{proof} 
We of course have $\mathfrak{i}^{(p)}(D_i) \subset [\mathfrak{i}^{(p)}(A_i)]_p \subset L^p(M_i)$, and so $L^p(D) \subset H^p(A)$.    The last assertion follows by topology 
since $E_p(H^p(A_i)) \subset L^p(D_i)\subset L^p(D)$ with $E_p$ restricting to the identity on $\mathfrak{i}^{(p)}(D_i)$.  \end{proof}

 In the language of e.g.\ \cite{GLa1} 
 the trace functional on $L^1(M_i)$ is ${\rm tr}(e_i a e_i)$, where  ${\rm tr}$ is the trace on $L^1(M)$.
Note that $M_i \subset {\mathfrak m}_\omega$. 
We see that ${\rm tr} = {\rm tr}_i$ on $M_i$.   In particular, ${\rm tr}(x) = \lim_i \, {\rm tr}(e_i x e_i) = \lim_i \, {\rm tr}_i (e_i x e_i)$.

We also have a trace $L^p$ duality.  If $x \in L^p, y \in L^q$ then $xy \in L^1$ and $e_i x e_i \to x$ in $p$-norm, $e_i y e_i \to y$ in $q$-norm,
$e_i x e_i y e_i \to xy$ in $1$-norm.  So $$\langle x , y \rangle = {\rm tr}(xy) = \lim_i \, {\rm tr}_i (e_i x y e_i) = \lim_i \, {\rm tr}_i (e_i x e_i y e_i).$$

As an easy consequence of the trace duality of noncommutative $L^p$-spaces we obtain the following:

\begin{proposition}  \label{Bek3.3ii}  For a semi-$\sigma$-finite subdiagonal algebra $A$ we have that
\begin{enumerate}
\item $H^p(A) = H^q_0(A)^\perp$ in $L^p(M)$ and $H^p_0(A) = H^q(A)^\perp$ for $1 < p,q < \infty$ with $\frac{1}{p}+\frac{1}{q}=1$;
\item $\mathscr{H}^q(A)  \cap \mathscr{H}^p(A) = \mathscr{H}^q(A)  \cap \mathscr{L}^p(A)$ and $\mathscr{H}^q_0(A)  \cap \mathscr{H}^p(A) = \mathscr{H}^q_0(A)  \cap \mathscr{L}^p(A)$ for all $ p, q \geq 1$.
\end{enumerate} 
\end{proposition} 
 
 \begin{proof}   Let $\xi \in L^p(M)$ with $\xi  \in H^q_0(A)^\perp$.   So using the duality above, $$\langle \xi , e_i \eta e_i \rangle = \langle e_i \xi e_i , e_i \eta e_i \rangle = 0, \qquad 
 \eta  \in H^q_0(A_i) .$$  By  the $\sigma$-finite case \cite[Corollary 3.4]{JiAnalytic}  we have $e_i \xi e_i \in H^q(A_i)$ so that $\xi = \lim_i \, e_i \xi e_i \in H^q(A)$.  The second 
assertion of (1) is similar.  

The proof of the second set of assertions  is as in \cite[Proposition 3.2 (ii)]{Bek}. \end{proof}  
 
  \begin{corollary}  \label{prs} If $A$ is a  semi-$\sigma$-finite subdiagonal algebra
 then $H^p(A) H^r(A) \subset H^s(A)$ if $1/s = 1/p + 1/r$ and $p,r,s \geq 1$.   \end{corollary} 

\begin{proof} 
 Let $x \in H^p(A) , y \in H^r(A)$.    Then $xy \in L^s(M)$ by e.g.\ \cite[Proposition 7.24]{GLa1}. 
 Suppose that $p, r < \infty$.
Let $x_i = e_i x e_i \to x$ in $p$-norm, and $y_i = e_i y e_i \to y$ in $r$-norm. 
Since $e_i x e_i \in H^p(A_i) \subset L^p(M_i), e_i y e_i \in H^r(A_i) \subset L^r(M_i
),$ and $A_i$ is maximal subdiagonal in  $\sigma$-finite $M_i$, we have 
$x_i y_i \in H^s(A)$ with norm there $\leq \| x_i   \| \| y_i \| \leq
K$.    Moreover $$\| x_i y_i - xy \|_s \leq  
\| x_i (y_i - y) \|_s + \|( x_i - x)y \|_s  
\leq \| x_i \|_p \| y_i - r \|_r + \| x_i - x \|_p \| y \|_s \to 0.$$ 
(That $x_i y_i \in H^s(A)$ is clear from such arguments and the density of $A_i$ in $H^p(A_i)$ for any $p \geq 1$.)  
Thus $xy$ is in $H^s(A)$ since the latter is closed in $L^p(M)$.  
Suppose that  $p = \infty$ and $r < \infty$, with  $y_i = e_i y e_i \to y$  in $r$-norm.  Then $e_j x y_i \to e_j xy$  in $r$-norm for each $j$.
  So $e_j xy \in H^r(A)$ from which it is easy to see
 that  $xy \in H^r(A)$. 
 \end{proof} 

In general the converse of the last result, Riesz factorization, is not even known for maximal subdiagonal algebras in the $\sigma$-finite case.
However if the $A_i$ do have Riesz factorization then so will $A$ by a weakly converging subnet argument, at least  for finite  $p,r > 1$.

 \begin{corollary}  \label{JiH}   {\rm (The Hilbert transform) \ }  Let $A$ be a semi-$\sigma$-finite subdiagonal algebra. 
For $1 < p < \infty$ there exists a 
bounded linear map $H : L^p(M) \to L^p(M)$ such that $x + iHx \in H^p(A)$ for all $x \in L^p(M)$.    This is the unique bounded extension of the map  $\mathfrak{i}^{(p)}(a + d + b^*) \mapsto i 
\, \mathfrak{i}^{(p)}(b^* - a)$ for $a, b \in A_0\cap\mathfrak{m}_\omega^M, d \in D\cap\mathfrak{m}_\omega^D$. \end{corollary} 

\begin{proof}  Indeed by \cite{JiAnalytic} we have unique 
norm bounded $H_i : L^p(M_i) \to L^p(M_i)$ such that $x + i H_i x \in H^p(A_i)$ for  all $x \in L^p(M_i)$, with $H_i$ extending the map $\mathfrak{i}^{(p)}(a + d + b^*) \mapsto i\mathfrak{i}^{(p)}(b^* - a)$ for $a, b \in (A_i)_0, d \in D_i$.    
By the uniqueness the $H_i$ are `compatible', hence they simultaneously extend to a bounded 
$H: L^p(M) \to L^p(M)$.  For $x \in L^p(M)$ we have   $x + i H x  = \lim_i \, e_i x e_i + i H_i(e_i x e_i) \in H^p(A)$.
Note that for $a, b \in A_0\cap\mathfrak{m}_\omega^M, d \in D\cap\mathfrak{m}_\omega^D$ $$H(\mathfrak{i}^{(p)}(a + d + b^*))  = \lim_i \, H(\mathfrak{i}^{(p)}(e_i a e_i + e_i d e_i + e_i b^* e_i)) = \lim_i \,  i \mathfrak{i}^{(p)}(e_i b^* e_i - e_i a e_i) =  i(b^* - a).$$
So by density $H$ is the unique bounded extension of the map  $a + d + b^* \mapsto i(b^* - a)$ for $a, b \in A_0\cap\mathfrak{m}_\omega^M, d \in D\cap\mathfrak{m}_\omega^D$. 
 \end{proof}

\begin{corollary}  \label{JiA3.3} Let $A$ be a semi-$\sigma$-finite subdiagonal algebra.  
For $1 < p < \infty$, $L^p(M) = H_0^p \oplus L^p(D) \oplus (H_0^p)^*$. 
\end{corollary} 

This is just as in Corollary 3.3 in \cite{JiAnalytic} if $p >1$.    Note that e.g.\ $H_0^p \cap (H_0^p)^* = (0)$ for $p \geq 1$, since if $x \in H_0^p \cap (H_0^p)^*$ then
$e_i x e_i \in H_0^p(A_i) \cap (H_0^p(A_i))^* = (0)$, so that $x = 0$. 

For $p = 1$ we have $L^1$-density: the $1$-closure of $H^1(A) + H^1(A)^*$ is $L^1(M)$ since it contains the unions of the $[\mathfrak{i}^{(1)}(A_i + A_i^*)]_1 = L_1(M_i)$.  In the last line we used 
$L^1$-density in the $\sigma$-finite case. 
(A proof of 
the latter: if $x \in M_i$ annihilates $\mathfrak{i}^{(1)}(A_i + A_i^*)$, then by \cite[Proposition 2.10]{GL2} 
$\mathfrak{i}^{(1)}(x)$ annihilates $A_i + A_i^*$, and hence also $M_i$ by weak* density. So $\mathfrak{i}^{(1)}(x) = 0$ whence 
$x = 0$.)
We have $H^1(A) = H^1_0(A) \oplus L^1(D)$.  Indeed suppose that  $x \in [A_0]_1 \cap [D]_1$ and $x_i = e_i x e_i \to x$  in the 1 norm. Our earlier analysis then shows that $(x_i)\subset 
[(A_i)_0]_1 \cap [D_i]_1$.
Then $0 = E(x_i) = x_i = E_p(x_i)$, whence $x = 0$.  So  $[A_0]_1 \cap [D]_1 = (0)$.

\section{A Beurling theory for semi-$\sigma$-finite subdiagonal subalgebras} \label{Beu} 

In the present context we define a {\em (right) $A$-invariant subspace} of $L^p(M)$, to be a 
closed subspace $K$ of $L^p(M)$ such that $K A \subset K$. For consistency, we will not 
consider left invariant subspaces at all, leaving the reader to verify that entirely symmetric 
results pertain in the left invariant case. Invariant subspaces may be classified in 
accordance with their structure. In this regard we say that an invariant subspace $K$ is 
{\em simply invariant} if in addition the closure of $K A_0$ is properly contained in $K$.

Given a right $A$-invariant subspace $K$ of $L^2(M)$, we define the {\em right wandering subspace} 
of $K$ to be the space $W = K \ominus [K A_0]_2$, and then say that $K$ is {\em type 1} if
$W$ generates $K$ as an $A$-module (that is, $K = [W A]_2$), and {\em type 2} if $W = (0)$.

The development of especially the $L^p$-version of the theory of closed right $A$-invariant subspaces, makes deep use of the concept of a `column $L^p$-sum' as introduced in \cite{JS}. Given $1 \leq p < \infty$ and a collection $\{ X_i : i \in I \}$ of closed subspaces of $L^p(M)$, the {\em external} column $L^p$-sum $\oplus^{col}_i \, X_i$ is defined to be the
closure of the restricted algebraic sum in the norm
$\Vert (x_i) \Vert_p \overset{def}{=} 
%zz
tr((\sum_i \, x_i^* x_i)^{\frac{p}{2}})^{\frac{1}{p}}$.
That this is a norm for $1 \leq p < \infty$ is verified in \cite{JiSaito}. If $X$ is a subspace of $L^p(M)$, and if 
$\{ X_i : i \in I \}$ is a collection of  subspaces of $X$, which together densely span $X$, with the added property that 
$X_i^* X_j = \{ 0 \}$ if $i \neq j$, then we say that $X$ is the {\em internal} column $L^p$-sum $\oplus^{col}_i \, X_i$. 
We shall not need the concept of an external column sum. So wherever column sum is mentioned below, it shall refer to an 
internal column sum. Note that  if $J$ is a finite subset of $I$, and if $x_i \in X_i\subset L^p$ for all $i \in J$, then we have that $$tr(|\sum_{i \in J} \, x_i|^p)^{1/p} = tr((|\sum_{i \in J} \, x_i|^2)^{\frac{p}{2}})^{1/p}
= tr((\sum_{i \in J} \, x_i^* x_i)^{\frac{p}{2}})^{1/p}.$$This shows that $X$ is then isometrically isomorphic to the external column $L^p$-sum $\oplus^{col}_i \, X_i$. Since the projections onto the summands are clearly
contractive, it follows by routine arguments (or by \cite[Lemma 2.4]{JiSaito}) that if $(x_i)\in \oplus^{col}_i \, X_i$, then the net $(\sum_{j \in J} \, x_j)$, indexed by the finite subsets $J$ of $I$, converges in norm to $(x_i)$.

The first cycle of results we present are extensions of corresponding results in \S 2 of \cite{LL}. 
The first result in this regard is basically a restatement of \cite[Theorem 2.4]{LL}. The exact same 
proof offered in \cite{LL} goes through in the general setting and hence we forgo the proof. 

\begin{theorem} \label{inv1}
Let $\A$ be an analytically conditioned algebra in $M$.    \begin{itemize}
\item [(1)]   Suppose that $X$ is a subspace of $L^2(M)$
of the form $X = Z \oplus^{col} [YA]_2$ where $Z, Y$ are closed
subspaces of $X$, with $Z$
a type 2 invariant subspace,
and $\{y^*x : y, x \in Y \} = Y^*Y \subset L^1(\D)$.  Then
$X$ is simply right $\A$-invariant if and only if $Y \neq
\{0\}$.
\item [(2)]  If $X$ is as in {\rm (1)},
then $[Y {D}]_2 = X \ominus [XA_0]_2$ (and
$X = [XA_0]_2 \oplus [Y {D}]_2$).
\item [(3)]   If $X$ is as described in {\rm (1)},
then that description also holds if $Y$ is replaced by $[Y \D]_2$.  Thus
(after making this replacement)
we may assume that $Y$ is a ${D}$-submodule of $X$.

\item [(4)]   The subspaces $[Y {D}]_2$ and $Z$ in the decomposition
in  {\rm (1)} are uniquely determined by $X$.  So is $Y$ if we
take it to be a ${D}$-submodule (see {\rm (3)}).
\item [(5)]  If $A$ is maximal subdiagonal, then any right $\A$-invariant subspace
$X$ of
$L^2(M)$ is of the form described in {\rm (1)},
with $Y$ the right wandering subspace of $X$.
\end{itemize}
\end{theorem}

Building on Theorem \ref{inv1}, we are now able to present the following rather elegant decomposition of the 
right wandering subspace. This extends \cite[Proposition 2.5]{LL}. The proof of the general case is quite 
a bit more tricky than that of the $\sigma$-finite case, and hence full details need to be provided. 

\begin{theorem} \label{newpr}  Let $A$ be an  analytically conditioned algebra in $M$. Suppose that $X$ is
as in Theorem {\rm
\ref{inv1}}, and that $W$ is the right wandering subspace of $X$.
Then $W$ may be decomposed as
an orthogonal direct sum $\oplus^{col}_i \, u_i L^2({D})$,
where $u_i$ are partial isometries in $\M$ for which 
$u_i(\frac{d\widetilde{\omega}}{d\tau_L})^{1/2}a\in W$ for each $a\in  A \cap \mathfrak{n}_\omega$, with
$u_i^* u_i \in {D}$, and $u_j^* u_i = 0$ if
$i \neq j$.   If $W$ has a cyclic vector for the ${D}$-action,
then we need only one partial isometry in the above.
\end{theorem}

\begin{proof}   By the theory of representations of a
von Neumann algebra (see e.g.\ the discussion at the
start of Section 3 in
\cite{JS}), any normal Hilbert
${D}$-module is an $L^2$ direct sum of cyclic
Hilbert ${D}$-modules,
and if $K$ is a normal
cyclic Hilbert ${D}$-module, then
$K$ is spatially isomorphic to $eL^2({D})$, for
an orthogonal projection $e \in {D}$.

Suppose that the latter isomorphism is implemented by a unitary ${D}$-module map $\psi$. Let $(f_\lambda)$ be a 
net in $\mathfrak{n}(D)_\omega^*\cap \mathfrak{n}(D)_\omega$ converging strongly to $\I$ (in a more general situation this is guaranteed by Lemma 9 in  \cite{Terp2}).  If in addition 
$K \subset W$, we will then have that $g_\lambda = \psi(e[f_\lambda h^{1/2}]) \in W$ for each $\lambda$, where 
$h=\frac{d\widetilde{\omega}}{d\tau_L}$. Then $$tr(d^* g_\lambda^* g_\lambda d) = \Vert \psi(e[f_\lambda h^{1/2}] d) \Vert_2^2 = tr(d^*(h^{1/2}f_\lambda^*) e[f_\lambda h^{1/2}] d),$$ for each $d \in {D}$, and so 
$g_\lambda^*g_\lambda=(h^{1/2}f_\lambda^*) e[f_\lambda h^{1/2}]=|e[f_\lambda h^{1/2}]|^2$. Hence there exists a partial 
isometry $u_\lambda$ majorised by $e$ such that $g_\lambda=u_\lambda e[f_\lambda h^{1/2}] =u_\lambda [f_\lambda h^{1/2}]$. 
By the modular action of $\psi$ we will then have that $\psi(e[f_\lambda h^{1/2}] d)= g_\lambda d=u_\lambda [f_\lambda h^{1/2}] d$ for any $d\in \D$. Since $L^2({D})$ is the closure of 
$\{(h^{1/2}d): d\in \mathfrak{n}({D})_\omega\}$ (see \cite[Proposition 7.40 \& Theorem 7.45]{GLa1}), and since 
$\psi(e[f_\lambda h^{1/2}] d)= u_\lambda [f_\lambda h^{1/2}] d =(u_\lambda f_\lambda)(h^{1/2} d)$ for each 
$d\in \mathfrak{n}({D})_\omega^*$, it follows $\psi(ef_\lambda b)=u_\lambda f_\lambda b$ for all $b\in L^2({D})$.

When working with ${D}$, we may of course assume that ${D}$ is in standard form, in which case the 
Haagerup-Terp standard form enables us to further identify $L^2(M)$ with the underlying Hilbert space of $M$. 
But then the $\sigma$-strong* convergence of $(f_\lambda)$ to $\I$ ensures that $ef_\lambda b$ will for any 
$b\in L^2({D})$ converge in $L^2$-norm to $eb$. Since $\|eb-ef_\lambda b\|_2 =\|\psi(eb-ef_\lambda b)\|_2=\|\psi(eb)-u_\lambda f_\lambda b\|_2$, this in turn ensures that $(u_\lambda f_\lambda b)$ converges to 
$\psi(eb)$ in $L^2$-norm. Given that the net $(u_\lambda f_\lambda)$ is in the unit ball of $M$, it must admit 
a subnet $(u_\gamma f_\gamma)$ which converges to some element $u_e$ of the unit ball of $M$. For any 
$b\in L^2(D)$ the net $(u_\gamma f_\gamma b)$ will then converge to $u_eb$ in the $L^2$-weak topology. But 
$(u_\gamma f_\gamma b)$ is also a subnet of $(u_\lambda f_\lambda b)$ which converges to $\psi(eb)$, and will 
therefore itself still be $L^2$-norm convergent to $\psi(eb)$. It is therefore clear that 
$\psi(eb)=u_eb$ for each $b\in L^2(D)$ and hence that $(u_\lambda f_\lambda b)$ is for each $b\in L^2(D)$, 
$L^2$-norm convergent to $u_eb$. For any $b\in L^2(D)$, we now also have that
\begin{eqnarray*}
tr(d^*b^*ebd) &=& \|ebd\|_2^2\\
&=& \lim_\lambda\|ef_\lambda bd\|_2^2\\
&=& \lim_\lambda\|\psi(ef_\lambda bd)\|_2^2\\
&=& \lim_\lambda\|\psi(ef_\lambda b)d\|_2^2\\
&=& \lim_\lambda\|u_\lambda f_\lambda bd\|_2^2\\
&=& \|u_ebd\|_2^2\\
&=& tr(d^*b^*u_e^*u_ebd).
\end{eqnarray*}
This equality firstly ensures that $b^*eb=b^*u_e^*u_eb$ for all $b\in L^2(\D)$, which then in turn ensures that 
$u_e^*u_e=e$. It follows that $u_e$ is a partial isometry with initial projection $e$, and that 
$\psi(eL^2({D}))=u_eL^2({D})$.

Given $u_i$ and $u_j$ with $i\neq j$, we have that $u_iL^2({D}), u_jL^2({D})\subset W$. 
Hence $L^2({D})u_j^* u_iL^2({D}) \subset L^1({D})$. Since  we have that  $tr([d_1^*h^{1/2}]u_j^* u_i(h^{1/2}d_0))$ equals
$$tr(\psi(e_j(h^{1/2}d_1))^*\psi(e_i(h^{1/2}d_0)))= tr([d_1^*h^{1/2}]e_je_i(h^{1/2}d_0))=0,$$ for 
any $d_0, d_1\in \mathfrak{n}({D})_\omega$, the density of 
$\{(h^{1/2}d): d\in \mathfrak{n}({D})_\omega \}$ in $L^2(D)$ now ensures that $u_j^* u_i=0$. In the case 
where $i=j$ we of course have that $u_i^*u_i=e_i\in {D}$. Putting these facts together,
we see that $W$ is of the desired form.
\end{proof}

The first corollary of the above theorem corresponds to \cite[Corollary 2.5]{LL}. Here too the proof of 
the general case requires more delicacy than that of the $\sigma$-finite case, and hence we state the proof in full.

\begin{corollary}  Let $A$ be a semi-$\sigma$-finite subdiagonal subalgebra.  Suppose that $X$ is
as in Theorem {\rm \ref{inv1}}, and that $W$ is the right wandering subspace of $X$.
If  $X\subset \mathcal{H}^2(A)$ then $Z \perp L^2({D})$. If additionally $A$ is maximal subdiagonal, 
then the partial isometries $u_i$ described in the preceding Proposition, all belong to $A$.
\end{corollary}

\begin{proof} If indeed $X\subset\mathcal{H}^2(A)$, it is a fairly trivial observation to make that 
$Z=[ZA_0]_2 \subset [XA_0]_2\subset [\mathcal{H}^2(A)A_0]_2=\mathcal{H}^2_0(A)$. Since 
$\mathcal{H}^2(A)=\mathcal{H}^2_0(A)\oplus L^2({D})$, the first claim follows. 

Now suppose that $A$ is maximal subdiagonal. To see the second claim recall that in the proof of Theorem \ref{newpr}, 
we showed that $u_i L^2({D})\subset W$ for each $i$.

Let $\{e_\alpha\}$ be the net of projections in $D$ increasing to $\I$  in Theorem \ref{decomp}. Given any $a_0\in e_\alpha A_0 e_\alpha \subset A_0$, we will for any 
$b\in e_\alpha L^2(D) e_\alpha = L^2(e_\alpha D e_\alpha)$ therefore have that $ a_0 e_\alpha u_i e_\alpha b \in a_0 W \subset \A_0H^2(A) \subset H^2_0(A) $. But $\mathbb{E}_2$ annihilates $H^2_0(A)$, and hence we must have that $0=\mathbb{E}_2(a_0 e_\alpha u_i e_\alpha b)=\mathbb{E}( a_0 e_\alpha u_i e_\alpha) b $ for all $b\in
 L^2(e_\alpha D e_\alpha)$. This can of course only be if $\mathbb{E}(a_0 e_\alpha u_i e_\alpha)=0$. Since $a_0\in e_\alpha A_0 e_\alpha$ was arbitrary, 
 we may now apply the sharpened Arveson maximality criterion in 
\cite[Theorem 2.2] {JiSaito} to see that $e_\alpha u_i e_\alpha \in e_\alpha A e_\alpha\subset \A$. The fact that $\{e_\alpha\}$ is increasing 
now clearly ensures that we in fact have that $e_\alpha u_i e_\beta \in A$ for any $\alpha$ and $\beta$. Therefore $u_i=\lim_\alpha\lim_\beta 
e_\alpha u_i e_\beta \in A$ as claimed.
\end{proof}

The next three results correspond to \cite[Corollary 2.6, Proposition 2.7 \& Theorem 2.8]{LL}. The proofs in 
\cite{LL} carry over to the general case, and hence we content ourselves with merely stating these results

\begin{corollary} \label{adcor}  If $X$ is an
invariant subspace of the form described in
Theorem {\rm \ref{inv1}}, then $X$ is type 1 if and only if
 $X = \oplus^{col}_i \, u_i \mathcal{H}^2(A)$, for $u_i$ as
in Theorem {\rm \ref{newpr}}.
\end{corollary}

\begin{proof}  If $X$ is type 1, then $X = [W A]_2$
where $W$ is the right wandering space, and so the one assertion
follows from Theorem {\rm \ref{newpr}}.
If $X = \oplus^{col}_i \, u_i \mathcal{H}^2(A)$, for $u_i$ as above, then
$[X A_0]_2 = \oplus^{col}_i \, u_i \mathcal{H}^2(A_0)$, and from this
it is easy to argue that  $W =
\oplus^{col}_i \, u_i L^2({D})$.
Thus $X = [W A]_2 = \oplus^{col}_i \, u_i \mathcal{H}^2(A)$.
  \end{proof}
  
\begin{proposition} \label{typestuff}   Let $X$ be a
closed $\A$-invariant subspace of $L^2(M)$, where $A$ is an
analytically conditioned subalgebra of $M$.
\begin{itemize}  \item [(1)]
If $X = Z \oplus [Y A]_2$ as in Theorem
 {\rm \ref{inv1}}, then $Z$ is type 2, and $[Y A]_2$ is type 1.
 \item [(2)]  If $A$ is a maximal subdiagonal algebra,
and if $X = K_2 \oplus^{col} K_{1}$ where
$K_1$ and $K_{2}$ are types 1 and 2 respectively,
then $K_1$ and $K_2$ are respectively the unique
spaces $Z$ and $[Y A]_2$ in  Theorem  {\rm \ref{inv1}}.
 \item [(3)]  If $A$ and $X$ are as in {\rm (2)},
and if $X$ is type 1 (resp.\ type 2),
then the space  $Z$ of Theorem  {\rm \ref{inv1}}
for $X$ is $(0)$ (resp.\ $Z = X$).
 \item [(4)]   If  $X = K_2 \oplus^{col} K_{1}$ where
$K_1$ and $K_{2}$ are types 1 and 2 respectively,
then the right wandering subspace for $X$
equals the right wandering subspace for $K_1$.
  \end{itemize} \end{proposition}

On collating the information contained in the preceding four results, we obtain the following structure theorem for invariant subspaces of $L^2$.

\begin{theorem} \label{main}  If $\A$ is a maximal subdiagonal
subalgebra of $M$, and if $K$ is a closed right  $A$-invariant subspace 
of $L^2(M)$,
then: \begin{itemize}
\item [(1)]   $K$ may be written uniquely as
an (internal) $L^2$-column sum $K_2 \oplus^{col} K_{1}$
of a type 1 and a type 2 invariant subspace of $L^2(M)$, respectively.
  \item [(2)]  If
$K  \neq (0)$ then $K$ is type 1 if and only if
$K = \oplus_i^{col} \, u_i \, H^2$, for $u_i$  partial isometries
with mutually orthogonal ranges and $|u_i| \in {D}$.
\item [(3)]
The right wandering subspace $W$ of $K$
is an $L^2({D})$-module in the sense of Junge and Sherman, and in
particular $W^* W \subset L^{1}({D})$.
\end{itemize}
\end{theorem}

The theory of invariant subspaces described above carries over to yield analogous statements in the setting of $L^p$ spaces ($1\leq p\leq \infty$) just as the authors 
showed in the $M$ `tracially finite' case in \cite{BL-Beurling}. However the proof of these facts requires some significant analysis which is far beyond the scope of this paper, and  will be postponed to a future paper of the second author and Xu \cite{LXdraft}.  
We content ourselves with merely stating the results. Details will be provided in the forthcoming paper of Labuschagne and Xu \cite{LXdraft}. The following result extends \cite[Corollary 4.3]{BL-Beurling}. 

\begin{theorem}\label{latt-isom} For any $1\leq p, q\leq \infty$ there is a lattice isomorphism from the weak*-closed right 
$A$-invariant subspaces of $L^p$ to those of $L^q$. We in particular have the following:
\begin{enumerate}
\item Given $2\leq p < \infty$ and a right $\A$-invariant closed subspace $K$ of $L^p(M)$, the prescription 
$K\to \overline{\mathcal{S}(K)_p}^{w^*}$ where $\mathcal{S}(K)_p$ is a right $A \cap \mathfrak{n}_\omega$-invariant subspace of 
$\mathfrak{n}_\omega$ for which $K=[\clo{\mathcal{S}(K)_ph^{1/p}}]_p$, realises a lattice isomorphism from the right 
$A$-invariant subspaces of $L^p(M)$ to those of $L^\infty(M)$.
\item Given $1\leq p < 2$ and a right $A$-invariant closed subspace $K$ of $L^p(M)$, the prescription 
$K\to [\clo{h^{1/2}\mathcal{S}(K)_p}]_2$ (where $r>0$ is chosen so that $\frac{1}{p}=\frac{1}{2}+\frac{1}{r}$ and where 
$\mathcal{S}(K)_p$ is a right $A \cap \mathfrak{n}_\omega$-invariant subspace of $\mathfrak{m}_\omega$ for which 
$K= [\clo{h^{1/2}\mathcal{S}(K)_ph^{1/r}}]_p$) realises a lattice isomorphism from the right $\A$-invariant subspaces of 
$L^p(\M)$ to those of $L^2(M)$.
\end{enumerate}
\end{theorem}

As in the case of $L^2$, we say that a closed right $A$-invariant subspace $K$ of $L^p$ is a type 2 invariant subspace if $K=[KA_0]_p$. With this concept in place, we are now able to present the following analogue of Theorem \ref{main}. This result extends each of \cite[Theorem 4.6]{sager} and \cite[Theorems 3.6 \& 3.8]{BekR}. The proof of the second part closely follows that of \cite[Theorem 4.5]{BL-Beurling}.

\begin{theorem} \label{lp-inv}  Let $\A$ be a maximal subdiagonal subalgebra of $\M$, and suppose that $K$ is a closed 
$\A$-invariant subspace of $L^p(M)$, for $1 \leq p \leq \infty$.  (For $p = \infty$ we assume that $K$ is $\sigma$-weakly closed.) 
\begin{enumerate}
\item The space $K$ may then be written as an $L^p$-column sum of the form $Z \oplus^{col} (\oplus_i^{col} \, u_i H^p)$ where $Z$ is a type 2 closed right $A$-invariant subspace of $L^p$, and where the $u_i$'s are partial isometries in $M \cap K$ with $u^*_ju_i = 0$ if $i \neq j$ and $u_i^* u_i \in \D$. Moreover, for each $i$, $u_i^* Z = (0)$, left multiplication by the $u_i u_i^*$ are contractive projections from $K$ onto the summands $u_i H^p(A)$, and left multiplication
by $1 - \sum_i \, u_i u_i^*$ is a contractive projection from $K$ onto $Z$.
\item Let $K$ be in the form $Z \oplus^{col} (\oplus_i^{col} \, u_i H^p)$ described above. Then there exists a contractive projection from $K$ onto $\oplus_i^{col} \, u_i L^p({D})$ and along $[KA_0]_p$. The quotient $K/[K A_0]_p$ is therefore isometrically ${D}$-isomorphic to the subspace $\oplus_i^{col} \, u_i L^p({D})$.
(Here $[\cdot]_\infty$ is as usual the $\sigma$-weak closure.)
\end{enumerate}
\end{theorem}

For $p \neq 2$ we say that a right $A$-invariant subspace of $L^p(M)$ is  type 1 if it corresponds to a type 1 invariant subspace in $L^2(M)$ under the canonical lattice isomorphism constructed
in Theorem \ref{latt-isom}. Then 
the first assertion of the preceding theorem may be interpreted as the statement that any closed right $\A$-invariant subspace of $L^p$ may be written as a column sum of a type 1 and type 2 invariant subspace. The following analogue of Beurling's characterization of weak* closed ideals of $H^\infty(\mathbb{D})$ now also readily follows from Theorem \ref{lp-inv}. This extends \cite[Corollary 4.8]{BL-Beurling} where this fact was noted for the case of finite von Neumann algebras.

\begin{corollary} \label{rid}  If  $\A$ is maximal subdiagonal, then the type 1 $\sigma$-weakly closed right ideals of $A$
are precisely those right ideals of the form $\oplus_i^{col} \, u_i A$, for partial isometries $u_i \in A$ with mutually 
orthogonal ranges and $|u_i| \in {D}$.
\end{corollary}

\section{Lebesgue decomposition, F \& M Riesz, and Gleason-Whitney theorems} \label{lebd} 

\subsection{F \& M Riesz spaces}

Let $M$ be a von Neumann algebra. 
Generalizing {\rm p.\ 8230} in {\rm \cite{blueda}},  a subalgebra (resp.\ subspace)  $A$ of a von Neumann algebra $M$ will be  said to be an {\em  F \& M Riesz algebra} (resp.\ {\em space})
 if whenever $\varphi \in M^*$ annihilates $A$ (that is, $\varphi \in A^\perp$) 
then the normal and singular parts, $\varphi_n$ and $\varphi_s$, also annihilate $A$.   
Following \cite{Ueda}, we showed in  \cite[Section 5]{blueda} that this F \& M Riesz property was a consequence of a {\em Lebesgue decomposition} for $A^*$.
This idea was generalized further to linear spaces in  {\rm \cite{ClouatreH}}.   
See e.g.\ Theorem 2.4 in \cite{ClouatreH}, although there are some subtle differences between their setup and ours. 

Write $A^*_s$ and $A^*_n$ for the set of restrictions to $A$ of singular and normal functionals on $M$. One has:

\begin{lemma} \label{fm} A  subspace $A$ of a von Neumann algebra $M$ is  an F $\&$ M Riesz space if and only if $A_n^* \cap A_s^* = (0)$,
and indeed  if and only if  $A$ has a unique Lebesgue decomposition  relative to $M$: that is, any $\varphi \in A^*$ may be written 
uniquely as $\varphi = \varphi_n + \varphi_s$
with $\varphi_n \in A^*_n$ and $\varphi_s \in A^*_s$.   Moreover, $\Vert \varphi \Vert 
= \Vert \varphi_n \Vert  + \Vert  \varphi_s \Vert$.  
 \end{lemma} 

\begin{proof}   Indeed if  $A_n^* \cap A_s^* = (0)$ suppose that  $\varphi \in M^*$ annihilates $A$.  Then $(\varphi_n)_{|A} = - (\varphi_s)_{|A}$ on $A$, so these equal 0.  
Conversely, if $A$ is an F \& M Riesz space and a functional in $A_n^* \cap A_s^*$ is the restriction of 
both $\varphi_n$ and $\varphi_s$, then $\varphi_n - \varphi_s$ annihilates $A$.
So $\varphi_n$ and $\varphi_s$ also annihilate $A$. 
Thus being an F \& M Riesz algebra is equivalent to  the assertion $A^* \cong A_n^* \oplus A_s^*$.   This is because any Hahn-Banach extension $\psi$ of 
a functional $\varphi \in A^*$ has a normal plus singular decomposition $\varphi_n + \varphi_s$.  Indeed as in \cite[Proposition 1(2)]{Ueda} we have 
$$\| \varphi \| = \| \psi \| = \| \varphi_n \| + \| \varphi_s \| \geq \| (\varphi_n)_A \| + \| (\varphi_s)_A \|  \geq  \| (\varphi_n)_A + (\varphi_s)_A \|  = \| \varphi \|. $$   
So $\Vert \varphi \Vert 
= \Vert \varphi_n \Vert  + \Vert  \varphi_s \Vert$.  
Conversely, the uniqueness of the Lebesgue decomposition implies that $A_n^* \cap A_s^* = (0)$.  \end{proof} 
 
It is well known that the F $\&$ M Riesz  property above implies the Gleason-Whitney property.  In an operator algebraic setting this first appeared in {\rm  \cite{BLue}}.
The proof there  (the last lines of \cite[Theorem 4.1]{BLue}
or p.\ 102 in \cite{BLsurvey}) generalizes vastly--see e.g.\ Remark 1 at the end of \cite{Ueda}, \cite[Section 5]{blueda}, or as  noted recently in \cite{ClouatreH}.  We state it in 
slightly greater generality.

\begin{corollary} \label{par}    F $\&$ M Riesz subspaces have the Gleason-Whitney property (GW1) from  {\rm  \cite{BLue}}:  every Hahn-Banach extension to $M$ 
of a relatively weak* continuous functional $\omega$ on an F $\&$ M Riesz subspace  $A$ of $M$  is normal.  \end{corollary}

\begin{proof}   This follows by a modification of the argument in the last lines of our proof of \cite[Theorem 4.1]{BLue}.  
Let $\varphi$ be a Hahn-Banach extension to $M$ of a normal functional $\omega$ on $A$, and write 
$\varphi = \varphi_n + \varphi_s$.  Let $\xi$ be the restriction to $A$ of $\varphi - \varphi_n$.  This is
relatively weak* continuous.  So by a known variant of the Hahn-Banach theorem it extends to a weak* continuous functional $\psi$ on $M$.
Since  $\varphi_s = \psi$ on $A$ we must have $\xi = 0$, thus 
 $\omega = \varphi_n$ on $A$.  But then $\| \varphi_n \| + \| \varphi_s \| = \| \varphi \| = \| \omega \| \leq \| \varphi_n \|$, so that 
$\varphi_s = 0$, and $\varphi$ is normal.  \end{proof} 

The following consequences are generalizations of results from \cite[Section 5]{blueda}
and others of our earlier papers (for example for the relation with (ii) in the next result  see the remark before 
\cite[Proposition 3.4]{BLvv}).

\begin{lemma} \label{gw2}  Suppose that $A$ is a unital subspace of a von Neumann algebra $M$.  
Consider the conditions: 
\begin{itemize} 
\item [(i)]  {\rm (GW2)}\ There is at most one normal Hahn-Banach extension to $M$ of any functional on $A$. 
\item [(ii)]  There is at most one normal state extension to $M$ of any state on $A$. 
\item [(iii)]   $A + A^*$ is weak* dense in $M$. 
\end{itemize} 
Then {\rm (i)} $\Rightarrow$ 
{\rm (ii)} $\Rightarrow$ {\rm (iii)}, and they all are equivalent  if in addition $A$ is a subalgebra of $M$. \end{lemma}  

\begin{proof}   Clearly  {\rm (i)} $\Rightarrow$ {\rm (ii)}, and that   (ii) $\Rightarrow$ {\rm (iii)} follows from the last paragraph of the proof of \cite[Lemma 5.8]{blueda}.   If $A$ is a subalgebra then   (iii) $\Rightarrow$ (i) by the other part of 
the proof of \cite[Lemma 5.8]{blueda}
(we note that in that proof  $\Vert  \varphi \Vert = 1$).
 \end{proof}

\begin{corollary} \label{co6}   Suppose that $A$ is an  F $\&$ M Riesz subspace of a von Neumann algebra $M$ such that  $A + A^*$ is weak* dense in $M$.     If   $\varphi \in M^*$
annihilates $A + A^*$ then $\varphi$ is singular.   Any normal functional on $M$ is the unique Hahn-Banach extension 
of its restriction to $A+ A^*$, and  in particular is normed by $A+ A^*$.  
\end{corollary}

\begin{proof} If $A$  is an  F $\&$ M Riesz subspace then so is $A^*$ using the fact that $\psi$ is normal (resp.\ singular)
iff $\psi^*$ is normal (resp.\ singular). The first assertion follows from this, as in \cite[Corollary 3.5]{BLue}.
This implies the last assertion  as in \cite[Theorem 4.2]{BLue}.   \end{proof} 

\begin{corollary} \label{co7}    If $A$ is  an  F $\&$ M Riesz subspace  of a von Neumann algebra
$M$ such that  $A + A^*$ is weak* dense in $M$, then  ${\rm Ball}(A + A^*)$ is weak* dense in ${\rm Ball}(M)$.
\end{corollary}

Actually a stronger `Kaplansky density' statement holds for F $\&$ M Riesz algebras, as in the remark after  \cite[Corollary 5.7]{blueda} referring to Theorem 3.2 there. 
The last lines of the proof of the latter theorem if $A$ is unital use Lemma 3.1 there, whose generality includes the present setting.

 The full `Gleason-Whitney  property' is a combination of both GW1 and GW2: 

\begin{corollary} \label{GWt} {\rm (\cite[Corollary 5.9]{blueda})}\    Suppose that $A$ is a unital F $\&$ M Riesz space in a von Neumann algebra
$M$.   Then $A$ has (GW2) if and only if every relatively weak* continuous  functional on $A$ has a unique  Hahn-Banach extension to $M$,  and if and only if every 
relatively weak* continuous 
functional on $A$ has a unique normal Hahn-Banach extension to $M$. 
(Recall if $A$ is also an algebra then (GW2) is equivalent to  $A+A^*$ being weak* dense in $M$.) \end{corollary}

There is an obvious variant of the last result with the word `functional' replaced by `state'.

\begin{corollary} \label{ucc}   {\rm (Cf.\ \cite[Section 3]{BLue} and \cite[Proposition 3.4]{BLvv})}\ Suppose that $A$ is a unital  F $\&$ M Riesz space in a von Neumann algebra
$M$. 
There is a  completely contractive (or equivalently, completely positive) normal extension to $M$ of a 
relatively weak* continuous completely  contractive unital $\Phi : A \to B(H)$.   Indeed any 
  contractive unital extension to $M$ of  a contractive unital relatively weak* continuous $\Phi : A \to B(H)$ is necessarily normal.  
  If  $A+A^*$ is weak* dense in $M$ then such extensions are 
  unique.
  \end{corollary}

\begin{proof} The idea from \cite[Proposition 3.4]{BLvv}:  Suppose that $\Psi  : M \to B(H)$ is a  contractive unital 
 extension of relatively weak* continuous contractive unital $\Phi : A \to B(H)$. 
The norm of $\langle \Phi( \cdot ) \zeta , \zeta \rangle$ is 1  if $\| \zeta \| = 1$.  Hence 
 $\langle \Psi( \cdot ) \zeta , \zeta \rangle$ is a Hahn-Banach 
extension.   So it is weak* continuous by GW1 (Corollary \ref{par}).   By polarization, $\Psi$ is weak* continuous.
  If  $A+A^*$ is weak* dense in $M$ then since $\Psi$ is completely positive and normal it   is the 
 necessarily unique weak* continuous extension of $\Phi$. 
\end{proof} 

\begin{corollary} \label{bdy}  Suppose that $A$ is a unital  F $\&$ M Riesz space in a von Neumann algebra 
$M$.    If  $A+A^*$ is weak* dense in $M$ then 
every $*$-representation of $M$ which is relatively weak* continuous on $A$ is a boundary representation of $A$ in the sense of {\rm 4.1.11} in 
\cite{BLM}, and is normal.  (This is Arveson's notion of boundary representation with the irreducibility requirement dropped; more recently 
referred to as the `unique extension property'.)   \end{corollary} 

\begin{proof}  Immediate from the last corollary.    
\end{proof} 

\begin{corollary} \label{pce} {\rm (Cf.\ \cite[Section 3]{BLvv} and \cite{BLvnce})}\  Suppose that $A$ is a unital  F $\&$ M Riesz space in a von Neumann subalgebra 
$M \subset B(H)$.  Suppose that $D$ is a von Neumann subalgebra  of $M$ which is contained in $A$ with $DAD \subset A$ (so $A$ is a $D$-bimodule). 
  If $\Phi : A \to D$ is a  relatively weak* continuous  contractive projection onto $D$, then $\Phi$ is completely  contractive, and extends  to a  
complete contraction $\Psi : M \to B(H)$, and this extension is necessarily normal.   
If  $A+A^*$ is weak* dense in $M$ then $\Psi$ is a normal conditional expectation onto $D$, and is unique.  \end{corollary} 

\begin{proof}  A contractive projection from $A$ onto  $D$  is a $D$-module map 
and hence is completely  contractive.  To see these  use \cite[Lemma  3.2]{BMag}, and the calculation two paragraphs above 
Lemma 1.1 in \cite{BLvv}).   Let $\Psi  : M \to B(H)$ be the completely contractive  (normal) extension 
in Corollary \ref{ucc}.   The rest is as in the proof of that result, and we have $\Psi(M) \subset D$ if  $A+A^*$ is weak* dense in $M$. \end{proof}

{\bf Remark.}  The above observations complement the work done in the later sections of \cite{BLvnce} where we studied conditions under which weak* continuous   contractive homomorphisms
$\Phi : A \to D$ extend to normal conditional expectations $M \to D$ (see for example \cite[Theorem 6.3]{BLvnce} in the case that  
 $M$ has a faithful normal tracial state.) 
 
 \medskip
 
 We remark that several results in this section may be generalized beyond the case that $M$ is a von Neumann algebra.  Indeed  some only use 
the fact that $M$ has a Lebesgue decomposition, and in some places also that $M$ is selfadjoint).

\subsection{Riesz  approximable  algebras} \label{Rapp} 

In this section we give many applications that not only apply to the semi-$\sigma$-finite subdiagonal algebras of 
Sections 2--4, but also in a more general setting that we call 
Riesz  approximable  algebras.   Moreover they apply in settings like  the Hardy space of the  half-plane 
 where the weight on $M$ is not semifinite on $D$, and to other much more general settings.   Indeed many of the results 
 here do not reference the specific weight on $M$,
their statements essentially only involving the Banach space structure of $A$ and $M$.

Suppose that $M$ has weak* closed $*$-subalgebras $M_i$  whose union is weak* dense in $M$, and that $A$ is a subalgebra (resp.\ subspace) of $M$ such that 
 $A_i = A \cap M_i$ is a F \& M Riesz algebra (resp.\ space)
 in $M_i$.
  We will also assume that 
the union of the $A_i$ is relatively weak* dense in $A$.   We also require that  the restriction of any positive singular functional on $M$ 
to $M_i$ is singular on $M_i$ (the word `compatible' is used for this property in \cite{ClouatreH}).   Note  that hereditary subalgebras have the latter property, as is easily seen 
using \cite[Theorem III.3.8]{tak1}.

We call an algebra (resp.\ space) $A$ with the properties in the last paragraph a {\em Riesz  approximable algebra} (resp.\ {\em space}) in $M$. 
One may then consider  Riesz  approximable algebras that are also subdiagonal, for example a weak* closed unital Riesz  approximable algebra in $M$ for which the $M_i$ are $\sigma$-finite and the 
subalgebras $A_i$ above are maximal subdiagonal in $M_i$ (which implies that $A_i$ is a F \& M Riesz algebra by \cite[Section 5]{blueda}).   These will include the semi-$\sigma$-finite maximal 
subdiagonal algebras in previous sections.

\begin{lemma} Every semi-$\sigma$-finite subdiagonal algebra is a Riesz  approximable  algebra. 
  \end{lemma}

\begin{proof} 
 Indeed as before, since  $e_i \nearrow 1$ strongly, $\cup_i \, M_i$ and $\cup_i \, A_i$ are weak* dense in $M$ and $A$ respectively. 
The $A_i$ are F \& M Riesz algebras by \cite[Theorem 5.3]{blueda}.   
Also, $M_i = e_i M e_i$ is a hereditary subalgebra of $M$, so as mentioned above 
 the restriction of any positive singular functional on $M$ to $M_i$ is singular on $M_i$.  
\end{proof} 

 {\bf Remark.} One of the important ingredients of the last proof is that subdiagonal algebras in a  $\sigma$-finite von Neumann algebra are F \& M Riesz algebras, or equivalently 
have a Lebesgue decomposition. 
This is proved in \cite[Theorem 5.3]{blueda} by following Ueda's strategy, which  follows Ando's strategy, of deducing the Lebesgue decomposition from an Amar-Lederer peak 
set type result.   We take the opportunity to make two related comments.  First, the proof of Proposition 2.3 in  \cite{blueda} needs correction (the result itself is true).  If $B$ is unital then in the
proof that (iii) implies (i) if $a_n \in B_{\rm sa}$ one may appeal to \cite[Proposition 3.11.9]{P} to see that $1-q$ is open.
If $b_n \searrow q$ then let $A = C^*(\{ 1, b_n \})$, a separable $C^*$-algebra with $q \in A^{**}$.   In $A^{**}$ the projection $q$ (resp.\ $1-q$) is compact so peak
(resp.\  is open, so is a support projection) with respect to $A$ by results  in the separable case from the first two papers 
of the first author and Read cited there.   In the last paragraph of the proof one may again appeal to the same Proposition
of the first author and Read (Proposition  6.4 (1) in  the 2013 paper) to see 
that a projection in $B^{**}$ which is peak with respect to $B^1$, is peak with respect to $B$. 

Amar and Lederer proved for $H^\infty$ of the disk that every `Lebesgue-null'  closed 
  sets  $E$ in the maximal ideal space $X$ of $M = L^\infty$, is contained in a null peak set in $X$ for $H^\infty$.     Ueda proved a noncommutative version of this for a restricted class of `closed sets' (rather, closed projections), namely the supports of singular functionals on $M$.    In \cite{blueda} we showed that 
this result of Ueda is still valid in a $\sigma$-finite von Neumann algebra $M$. 
Recently the first author and Clou\^atre have obtained the full noncommutative Amar and Lederer  theorem 
(i.e.\ for all closed null projections, as opposed to supports of singular functionals) for subdiagonal algebras in such a  von Neumann algebra $M$.

\bigskip

The Lebesgue decomposition, F \& M Riesz, and Gleason-Whitney theorems all work for Riesz  approximable  algebras.

\begin{theorem} \label{Lebde} A Riesz  approximable algebra  (resp.\ space) 
$A$ is an F \& M Riesz algebra  (resp.\ space) 
and has a unique Lebesgue decomposition.   Thus it has all the properties of 
F \& M Riesz spaces from the last section, such as the Gleason-Whitney theorems and corollaries there.  
In particular, semi-$\sigma$-finite subdiagonal algebras have all the properties mentioned in statements of the results in the last section.
\end{theorem}

\begin{proof} 
Suppose that  $\varphi \in A_n^* \cap A_s^*$, in the notation above Lemma \ref{fm}.   If $\psi$ is  singular on $M$ then it is singular on 
$M_i$.    We see that  $\varphi_{|A_i}$ is in $(A_i)^*_s$ and $(A_i)^*_n$, so is zero by hypothesis. 
Since $\varphi$ is the restriction of a normal functional  and $\cup_i \, A_i$ is relatively weak* dense in $A$ we have that 
$\varphi = 0$.   
 \end{proof} 
 
 We may also appeal to several other known consequences of a Lebesgue decomposition.  For example: 

 \begin{corollary}  \label{pf} The predual $A_*$ of a weak* closed Riesz  approximable  algebra  $A$  
is an $L$-summand in $A^*$. 
  Also, $A_*$ has property {\rm  (V$^*$)} and is weakly sequentially complete.  Hence if $A_*$ is separable then it is a unique predual of $A$. 
 \end{corollary} 

\begin{proof}  The first assertion is obvious from the theorem, and the second and third follow from results of Pfitzner  and Pelczynski as in   \cite[Corollary 2]{Ueda}. 
The last assertion is then the main result in \cite{Pf}, 
also due to  Pfitzner. \end{proof} 

{\bf Remark.} The condition that $A_*$ be separable to be a unique predual is possibly not necessary. 
The 
proof  in \cite{Ueda} (and reprised in \cite{blueda} in the $\sigma$-finite setting) for uniqueness of predual uses  Property (X) of Godefroy and Talagrand, which is a sequential criterion.

\medskip

We conclude by mentioning some other Hardy space properties of our algebras.
 Clearly semi-$\sigma$-finite subdiagonal algebras, and more generally Riesz  approximable algebras, will have the 
UNSEP (the unique normal state extension property).    We know that they will not in general be 
Ueda algebras in the sense of \cite[Section 5]{blueda} (assuming that measurable cardinals exist), since it is shown in 
\cite[Theorem 6.1]{blueda} that then semifinite commutative von Neumann algebras need not be 
Ueda algebras.   However  semi-$\sigma$-finite subdiagonal algebras are of course an `increasing  limit' of Ueda algebras.    Similarly, we do not know if they have `factorization'.  Indeed  in the 
$\sigma$-finite case authors have only able to 
obtain  a one-sided partial factorization \cite{JiSaito,Ji}
(see also \cite{Bek2}).   Similarly it is not likely that in general $C^*_e(A) = M$ as we have in the tracially finite case; indeed  it seems 
that probably $A$ does not even generate $M$ as a $C^*$-algebra in the general case.  If the latter fails, 
then we also will not have the uniqueness of extension of completely contractive homomorphisms property in \cite[Theorem 8.3]{BLsurvey}, a property that does hold
in the tracially finite case covered in that reference.  

It is to be expected that some of the properties  proved in our paper 
will actually turn out to be equivalent, and characterize semi-$\sigma$-finite subdiagonal algebras within
the  analytically conditioned algebras,  similarly to the characterizations of 
 subdiagonal algebras in \cite{BLChar,BLsurvey} and some of the other papers cited above.


\begin{thebibliography}{99}

 


\bibitem{Arv}    W. B. Arveson, {\em Analyticity
in operator algebras,} Amer.\ J.\ Math.\ {\bf 89} (1967),
578--642.

\bibitem{Bek} T. N. Bekjan, {\em Noncommutative Hardy space associated with semi-finite subdiagonal algebras,} J.\ Math.\ Anal.\  Appl.\ {\bf 429} (2015), 1347--1369.

\bibitem{Bek2} T. N. Bekjan, {\em Szeg\"o type factorization of Haagerup noncommutative Hardy spaces}, Acta Mathematica Scientia {\bf 37B}(5) (2017), 1221--1229.

\bibitem{BO} T. N. Bekjan and A. Oshanova, {\em  Semifinite tracial subalgebras,} Ann. Funct. Anal. {\bf 8} (2017), 473--478.

\bibitem{BekR}  T. N. Bekjan and M. Raikhan, {\em A Beurling-Blecher-Labuschagne type theorem for Haagerup
noncommutative $L^p$ spaces,} Banach J. Math. Anal. {\bf 15} (2021), paper no.\ 39, 24 pp.


\bibitem{BLue}  D. P. Blecher and L. E. Labuschagne, 
{\em  Noncommutative function theory and unique extensions}, 
 Studia Math. {\bf 178} (2007), 177-195.



\bibitem{BLsurvey}
  D. P. Blecher and L. E. Labuschagne, {\em Von Neumann algebraic $H^p$ theory,} 
 Function Spaces: Fifth Conference on Function Spaces,
Contemp. Math. Vol.\ 435, Amer.\ Math.\ Soc.\ (2007).
  
  \bibitem{BLChar} D. P. Blecher and L. E. Labuschagne, {\em Characterizations of noncommutative $H^\infty$,} J.\ 
  Integral Eq.\ and Operator Th.\ {\bf 56} (2006), 301--321.
  
\bibitem{BLvv}  D. P. Blecher and L. E. Labuschagne, {\em On vector-valued characters for noncommutative function algebras},
 Complex Anal.\ Oper.\ Theory {\bf 14} (2020),  Paper No.\ 31. 


\bibitem{BL-Beurling} D. P. Blecher and L. E.  Labuschagne, {\em A Beurling theorem
for noncommutative $L^p$,} J.\ Operator Theory \textbf{59} (2008), 29-51.

\bibitem{blueda}  D. P. Blecher and L. E. Labuschagne, {\em Ueda's peak set theorem for general von Neumann algebras,} Trans. Amer. Math. Soc. {\bf 370} (2018), 8215--8236.

\bibitem{BLvnce}  D. P. Blecher and L. E. Labuschagne, {\em Von Neumann algebra conditional expectations with applications to generalized representing measures for noncommutative function algebras,} Adv.\ in Math.\
{\bf 396} (2022), doi 10.1016/j.aim.2021.108104

\bibitem{BLM}  D. P. Blecher and 
C.  Le Merdy, {\em Operator algebras and their modules---an
operator space approach,} Oxford Univ.\  Press, Oxford (2004).

\bibitem{BMag}  D. P. Blecher and  B. Magajna, {\em Duality and operator algebras: Automatic weak* continuity and applications}, J. Funct.\ Analysis {\bf
 224} (2005), 386--407.

\bibitem{ClouatreH} R. Clou\^atre and M. Hartz, {\em   Lebesgue decompositions and the Gleason-Whitney property for operator algebras,} Preprint (2022), 	arXiv:2211.04366. 


\bibitem{CT} R. Clou\^atre and E. Timko, {\em   Non-commutative measure theory: Henkin and analytic functionals on $C^*$-algebras,} arXiv:2105.11295, to appear  Math. Ann. (2022).


\bibitem{Godefroysurvey} G. Godefroy, {\em Existence and uniqueness of isometric preduals: a survey,}
 Banach space theory (Iowa City, IA, 1987), 131–193, Contemp. Math., 85, Amer. Math. Soc., Providence, RI, 1989. 

\bibitem{Gold}  S. Goldstein, {\em  Conditional expectation and stochastic integrals in noncommutative $L^p$ spaces,} Math.\ Proc.\ Cambridge Philos.\ Soc.\
\textbf{110} (1991),  365--383.

\bibitem{GoLab} S. Goldstein and L. E. Labuschagne, {\em Composition Operators on Haagerup $L^p$-spaces,}
\textit{IDAQP}, \textbf{12} (2009), 439--468.

\bibitem{GLa1}  S. Goldstein \& L. E.  Labuschagne, {\it Notes on noncommutative  $L^p$ and Orlicz spaces,} {\L}odz University Press, {\L}odz, 2020. ISBN 978-83-8220-385-1, e-ISBN 978-83-8220-386-8.

\bibitem{GLa2}  S. Goldstein and L. E. Labuschagne, Book in preparation.

\bibitem{GL2} S. Goldstein and J. M. Lindsay, {\em Markov semigroups
KMS-symmetric for a weight,} Math. Annalen \textbf{313} (1999), 39--67.

\bibitem{HaaSf}  U. Haagerup, {\em The standard form of von Neumann algebras,} Math. Scand. {\bf 37} (1975),  271--283. 

  \bibitem{HJX}    U. Haagerup, M. Junge and Q.  Xu, {\em A reduction method for noncommutative $L_p$-spaces and applications,}  Trans. Amer. Math. Soc. {\bf 362} (2010), 2125--2165. 



\bibitem{Ji}   G-X. Ji, {\em A noncommutative version of $H^p$ and characterizations of subdiagonal subalgebras,}
Integr.\ Eq.\  Oper.\ Theory {\bf 72} (2012), 131--149.

\bibitem{JiAnalytic}   G-X. Ji, {\em  Analytic Toeplitz algebras and the Hilbert transform associated with a subdiagonal
algebra,} Sci.\ China Math {\bf 57} (2014), 579--588.

\bibitem{JOS}   G-X. Ji,  T. Ohwada, and K-S. Saito, {\em Certain structure of subdiagonal algebras,} J. Operator Theory {\bf 39} (1998),  309–317.

\bibitem{JiSaito}   G-X. Ji and K-S. Saito, {\em Factorization in Subdiagonal Algebras}, J.\ Funct.\ Analysis {\bf159} (1998),
191--201.

\bibitem{JS}   M. Junge and D. Sherman, {\em Noncommutative $L^p$ modules,} J. Operator Theory {\bf 53} (2005), 3--34. 

\bibitem{JX}   M. Junge and Q. Xu, {\em Noncommutative Burkholder-Rosenthal inequalities,} Annals of Prob., {\bf 31} (2003), 948--995.

\bibitem{LL}  L. E. Labuschagne, {\em Invariant subspaces for $H^2$-spaces of $\sigma$-finite algebras,} Bulletin LMS,
{\bf 49} (2017), 33--44.

\bibitem{L-Orlicz} L. E. Labuschagne, {\em A crossed product approach to Orlicz spaces,} Proc. London Math. Soc.  \textbf{107} (2013), 965--1003.

\bibitem{LXdraft}  L. E. Labuschagne and Q. Xu,  {\em An extension of Haagerup's reduction theorem with applications to subdiagonal subalgebras of general von Neumann algebras,} in preparation.


\bibitem{P} G. K. Pedersen, {\em C*-algebras and their automorphism
groups,} Academic Press, London (1979).

\bibitem{PT} G. K. Pedersen and M. Takesaki, {\em The Radon-Nikodym theorem for von Neumann algebras,} 
Acta Math. {\bf 130} (1973), 53--87. 

\bibitem{Pf}  H. Pfitzner, {\em Separable L-embedded Banach spaces are unique preduals,} Bull. Lond. Math. Soc. {\bf 39}  (2007), 1039--1044.

\bibitem{sager} L. Sager, {\em A Beurling-Blecher-Labuschagne theorem for noncommutative Hardy spaces associated with semifinite von Neumann algebras,}
 Integral Equations Operator Theory {\bf 86} (2016),  377--407. 

\bibitem{Stratila} S. Stratila, {\em Modular theory in operator algebras,} Second edition, Cambridge-IISc Series. Cambridge University Press, Delhi, 2020. 

\bibitem{tak1}   M. Takesaki, {\em Theory of Operator Algebras I},
Springer, New York, 1979.

\bibitem{tak2}   M. Takesaki, {\em Theory of Operator Algebras II},  Encyclopaedia of Mathematical Sciences, Vol.\ 125,   Springer-Verlag, Berlin, 2003. 



\bibitem{Terp2}   M. Terp, {\em Interpolation spaces between a von Neumann algebra and
its predual}, J. Operator Theory, {\bf 8}(2)(1982), 327-360.

\bibitem{Ueda}   Y. Ueda,
{\em On  peak phenomena for non-commutative $H^\infty$}, Math. Ann. {\bf 343}
 (2009),  421--429.
 
\bibitem{Wat} K. Watanabe, {\em Dual of non-commutative $L^p$-spaces with $0 < p < 1$}, Math. Proc. Camb. Phil. Soc. \textbf{103}(1988), 503-509.

\bibitem{Xu}  Q. Xu, {\em On the maximality of subdiagonal algebras,}
J. Operator Th.\ {\bf 54} (2005), 137--146.

\end{thebibliography}
\end{document}